\documentclass[11pt]{article}
\usepackage{amsmath,amsthm,amssymb,ifsym,a4wide} 
\usepackage[twoside]{geometry}
\usepackage{graphicx}
\usepackage{epsfig} 
\usepackage{slashed}  
\usepackage[hidelinks]{hyperref}
\newtheorem{theorem}{Theorem}[section]
\newtheorem{proposition}[theorem]{Proposition}
\newtheorem{lemma}[theorem]{Lemma}

\numberwithin{equation}{section} 
\numberwithin{figure}{section}  
\usepackage{amssymb, amscd, mathrsfs, wasysym} 

\newcommand \la \langle
\newcommand \ra \rangle

\newcommand \Kcal {\mathcal{K}}
\newcommand \Hcal {\mathcal{H}}

\newcommand \underdel {\underline \partial}

\newcommand \tildeU {\widetilde U}

\newcommand \trianglerightNEW \triangleright

\newcommand \auth {\textsc}

\newcommand \bei {\begin{itemize}}
\newcommand \eei {\end{itemize}}
\newcommand \be {\begin{equation}}
\newcommand \bel {\be\label}
\newcommand \ee {\end{equation}}
\newcommand \del \partial
\newcommand \RR {\mathbb R}

\newcommand \eps \epsilon

\usepackage[most]{tcolorbox} 

\let\oldmarginpar\marginpar
\renewcommand\marginpar[1]{\-\oldmarginpar[\raggedleft\footnotesize #1]%
{\raggedright\footnotesize #1}}
\begin{document}

\title{\bf \Large 
Stability of a coupled wave-Klein-Gordon system with quadratic nonlinearities
} 

\author{Shijie Dong\footnote{Fudan University, School of Mathematical Sciences; Sorbonne Universit\'e, CNRS, Laboratoire Jacques-Louis Lions, 4, Place Jussieu, 75252 Paris, France. 
Email: dongs@ljll.math.upmc.fr.}
\, 
and Zoe Wyatt\footnote{Maxwell Institute for Mathematical Sciences, School of Mathematics, University of Edinburgh, Edinburgh, EH9 3FD, United Kingdom. 
Email: zoe.wyatt@ed.ac.uk.}}


\date{\today}
\maketitle

\begin{abstract} 
{
Using the hyperboloidal foliation method, we establish stability results for a coupled wave-Klein-Gordon system with quadratic nonlinearities. In particular, we investigate  quadratic wave-Klein-Gordon interactions in which there are no derivatives on the massless wave component. By combining hyperboloidal energy estimates with appropriate transformations of our fields, we are able to show global existence of solutions for sufficiently small initial data. Our result also proves small data global existence of the Klein-Gordon-Zakharov equations using the hyperboloidal foliation. 
}
\end{abstract}



\section{Introduction} 
\paragraph{Model problem.}
Systems of wave equations and Klein-Gordon equations are of great importance in mathematics and physics.  Examples in the field include the semilinear Dirac-Proca equations and Klein-Gordon-Zakharov equations, and the quasilinear Einstein-Klein-Gordon equations. In this paper we will study the following semilinear coupled wave-Klein-Gordon system using the hyperboloidal foliation method of LeFloch-Ma \cite{PLF-YM-book}. 
Consider\footnote{$\Box := \eta^{\alpha \beta} \del_\alpha \del_\beta$, with $\eta = \text{diag} (-1, 1, 1, 1)$. Unless specified, Roman letters and Latin letters take values in $\{0, 1, 2, 3\}$ and $\{1, 2, 3\}$ respectively, and Einstein summation convention is adopted.}
\bel{eq:model}
\aligned
-\Box u 
&= F_u :=
u v + u \del_t v ,
\\
-\Box v + v 
&=  F_v :=
u v,
\endaligned
\ee
with initial data prescribed on the time slice $t=2$
\bel{eq:model-ID}
\aligned
\big( u, \del_t u \big) (t=2, \cdot)
&=
\big( u_0, u_1 \big),
\\
\big( v, \del_t v \big) (t=2, \cdot)
&=
\big( v_0, v_1 \big).
\endaligned
\ee

 Our aim is to prove that initial data, sufficiently small in some norm, yield global-in-time solutions that decay back to the trivial solution. The main difficulty is that there are no derivatives on the wave component $u$ on the right-hand-side terms $F_u$ and $ F_v$ of equation \eqref{eq:model}, and thus the nonlinearities appear to decay insufficiently fast. To be more precise, the best we can expect is that
\be 
\| F_u \|_{L^2}
=
\| u v + u \del_t v \|_{L^2}
\lesssim t^{-1},
\qquad
\| F_v \|_{L^2}
=
\| u v \|_{L^2}
\lesssim t^{-1},
\ee
both of which are not integrable. For convenience, here and throughout the paper we use $A \lesssim B$ to indicate $A \leq C B$ with $C$ a generic constant.

\paragraph{Previous work and motivation.} Before we demonstrate our techniques for treating \eqref{eq:model}, let us briefly discuss some previous work in the literature. 
Recall, in the celebrated counterexample by John \cite{John}, that there exist wave equations with certain nonlinearities that are quadratic in derivatives but which do not admit global-in-time solutions. Nonetheless, a broad class of wave equations with nonlinearities, quadratic in derivatives, satisfying the so-called null condition, as shown independently by Klainerman \cite{Klainerman86} and Christodoulou \cite{Christodoulou}, do admit global-in-time solutions. The vector field method, due to Klainerman, and the conformal method, due to Christodoulou, have been two major approaches to studying wave equations. 
Other related versions of the null condition have also been used to great effect, see for example \cite{L-R-cmp, L-R-annals} and \cite{P-S-cpam}.

By contrast, the Klein-Gordon equation requires a different analysis from the wave equation. One key obstruction is that the scaling vector field $S = t\del_t + x^a \del_a$ does not commute with the Klein-Gordon operator $- \Box + 1$, which thus prevents us from applying the Klainerman-Sobolev inequality directly. Pioneering works by Klainerman using the vector field method in \cite{Klainerman85}, and by Shatah employing a normal form method in \cite{Shatah}, led the way in treating a wide class of Klein-Gordon-type equations. 

Furthermore our study of the PDE \eqref{eq:model} was motivated by other coupled wave-Klein-Gordon systems in the literature and future work on Dirac-Klein-Gordon systems \cite{DLW}. For example, Tsutsumi and his collaborators studied the Dirac-Proca system in \cite{Tsutsumi} and the Klein-Gordon-Zakharov system in \cite{OTT}. Katayama also investigated a coupled wave-Klein-Gordon system with a large class of quadratic nonlinearities in \cite{Katayama12a}. With these in mind, our aim is to utilise the hyperboloidal foliation method developed by LeFloch and Ma in \cite{PLF-YM-cmp}, where the authors studied a quasilinear coupled wave-Klein-Gordon system. See also the work of Wang \cite{Wang} and Ionescu and Pausader \cite{Ionescu-P} for other efforts in this direction.

\paragraph{Main result.}
Returning to our system \eqref{eq:model}, we find that we can treat the $uv$ nonlinearity appearing in $F_u$
by transforming the variable $u$ in a similar way to the work of Tsutsumi in \cite{Tsutsumi}. Note this is only at the expense of bringing a null form into the new wave equation. 
As for the nonlinear term $u \del_t v$ in $F_u$, we rewrite it as two terms $u \del_t v = \del_t (u v) - v\del_t u $, in which the former is a total derivative and the latter is easier to deal with due to the derivative on the wave component. Then, following \cite{Katayama12a}, we split the wave equation into two new wave equations, and the strategy for handling the $u v$-type nonlinearity applies once more.
On the other hand, to treat the $uv$ term appearing in $F_v$ of the Klein-Gordon equation, the novel idea is that we move the term to the left hand side and treat $v$ as a Klein-Gordon field with varying mass $m = \sqrt{1-u}$. This enables us to apply the techniques in \cite{PLF-YM-cmp}. 

We are now ready to state the main theorem.

\begin{theorem}[Nonlinear stability of a wave-Klein-Gordon model]
\label{thm:wKG}
Consider the  system \eqref{eq:model} 
and let $N \geq 8$ be an integer.  
Then there exists $\eps_0 > 0$ such that for all $\eps \in (0, \eps_0)$ and 
all compactly supported initial data $(u_0, u_1, v_0, v_1)$
satisfying the smallness condition 
\bel{main-thm-initial-data}
\| u_0, v_0 \|_{H^{N+1}(\RR^3)} + \| u_1, v_1 \|_{H^N(\RR^3)} 
\leq \eps,
\ee
the initial value problem \eqref{eq:model}--\eqref{eq:model-ID} admits a global-in-time solution $(u, v)$ with
\be
| u(t, x) | 
\lesssim t^{-1},
\quad
| v(t, x) |
\lesssim t^{-3/2}.
\ee 
\end{theorem}

In the proof of Theorem \ref{thm:wKG}, we assume the initial data are prescribed on the slice $t = 2$, and that the initial data are supported in the unit ball $\{x : |x| \leq 1 \}$. These assumptions are not essential by noting the following 1) if the initial data are prescribed on the slice $t = T$, then we can translate the time $t$ to the new time variable $\tau = t - T + 2$, and now the initial time becomes $\tau = 2$ with everything else unchanged; 2) if the initial data are supported in the ball $\{ x : |x| \leq R \}$ with $R > 1$, then we can rescale the spacetime variables $(t, x)$ to new variables $(t', x') = (1/R) (t, x)$, under which the initial data are supported in the unit ball and the structure of the equations remain the same.

For the proof of the main theorem, we employ the strategy introduced by LeFloch and Ma in \cite{PLF-YM-cmp}, which allows us to obtain robust pointwise decay for both wave and Klein-Gordon components. We also apply a hyperboloidal conformal-type energy estimate for the wave component, which was first introduced by Ma and Huang in \cite{YM-HH}. This  enables us to obtain  good $L^2$-type bounds for the wave component $u$. Altogether, our proof is shorter and yields better energy bounds for both wave and Klein-Gordon components compared to those in \cite{PLF-YM-cmp}. 
It is in principle possible to remove our restriction to compact initial data, see for example \cite{PLF-YM-arXiv1, PLF-YM-arXiv2} or \cite{Klainerman-QW-SY}.


\paragraph{The Klein-Gordon-Zakharov equations.}
Our result allows us to also deal with the Klein-Gordon-Zakharov equations, which have been studied before using constant time slices or phase-space methods in \cite{OTT, Katayama12a, Tsutaya}. Moreover we could also treat some Dirac-Klein-Gordon and Dirac-Proca type equations but we will discuss this in future work \cite{DLW}.
Recall the Klein-Gordon-Zakharov equations
\bel{eq:KGZ} 
\aligned
&- \Box u 
= \sum_a \Delta |v_a|^2,
\\
&- \Box v_a + v_a
= u v_a,
\endaligned
\ee
where the unknown $u$ is real valued and $v_a$ are complex valued for $a=1,2,3$. The initial data are denoted by
\be 
\big( u, \del_t u \big) (t=2, \cdot)
= \big( u^{(0)}, u^{(1)} \big),
\qquad
\big( v_a, \del_t v_a \big) (t=2, \cdot)
= \big( v_a^{(0)}, v_a^{(1)} \big),
\ee
In order to apply the strategy of Theorem \ref{thm:wKG}, we rewrite the equations \eqref{eq:KGZ} in the following form
\bel{eq:KGZ-form2}
\aligned
&- \Box u 
= \sum_a \del_i \Big( \del^i \big(x_a^2 + y_a^2\big) \Big),
\\
&- \Box x_a + x_a
= u x_a,
\\
&- \Box y_a + y_a
= u y_a,
\endaligned
\ee
in which we use the notations $x_a := \text{Re} (v_a)$ and $y_a := \text{Im} (v_a)$ to denote the real part and imaginary part of a complex number $z$, respectively.

We note that the regularity of $u$ is one order less than that of $v_a$. This can be seen from the initial data, which we consider the norm
$$
\| u^{(0)} \|_{H^{N_0}},
\quad
\| u^{(1)} \|_{H^{N_0-1}},
\quad
\| v_a^{(0)} \|_{H^{N_0+1}},
\quad
\| v_a^{(1)} \|_{H^{N_0}},
$$
with $N_0$ some large integer. Thus equations \eqref{eq:KGZ} are semilinear equations. Note also that the wave nonlinearity in \eqref{eq:KGZ-form2} is of divergence form, and thus easier to handle than that those in Theorem \ref{thm:wKG}. Thus our method of proof applies to this system in a very similar way and for which reason we omit the details.

\paragraph{Further generalisations of the main result.} One can also show using our methods, though we will not explicitly do so here, that Theorem~\ref{thm:wKG} is also true for the following more general system
\bel{eq:general} 
\aligned
-\Box u 
&= 
Q(u, v, \del v; v, \del v),
\\
-\Box v + v 
&= 
Q(u; u, v) + Q(\del u, v, \del v; v, \del v),
\endaligned
\ee
where we use the short-hand notation $Q(\cdots ; \cdots)$ to denote quadratic nonlinearities involving interactions between one term from each side of the semicolon. Note further that compared to the work of \cite{OTT} and \cite{Katayama12a} a wider class of nonlinearities can be treated. In \cite[(2.14)]{Katayama12a} any nonlinearity for the wave equation involving at most one derivative, needed to be of divergence form. This is not needed in our setting. 
We also remind one that the $u$--$u$ interaction term above was treated by Tsutsumi in \cite{Tsutsumi}. Finally, it was speculated in \cite{PLF-YM-book} that  nonlinear interaction terms of the form 
$$ 
Q(u; v, \del v)
$$ 
may lead to finite time blow-up. Thus this article partially answers their question by showing that certain terms of this form do not lead to finite time blow-up. See the discussion in Section \ref{sec:Close} for further comments on the generalisation of our main result. 

\subsection*{Outline}

The rest of this article is organised as follows.
In Section~\ref{sec:BHFM}, we revisit the basics of the hyperboloidal foliation method; next, the estimates for commutators and null forms are given in Section~\ref{sec:ECNF}; later on, we illustrate the techniques  obtaining pointwise decay estimates for wave and Klein-Gordon components in Section \ref{sec:P-WKG}; in Section~\ref{sec:BM}, by initialising the bootstrap method, we provide some basic estimates needed afterwards; we then derive refined estimates for Klein-Gordon and wave components in Section~\ref{sec:refineKG} and Section~\ref{sec:refinew} respectively; in the last section, we demonstrate the proof of the main theorem, and give some remarks.


\section{Basics of the hyperboloidal foliation method}
\label{sec:BHFM}
 
\subsection{Hyperboloidal foliation of Minkowski spacetime}
 
In order to introduce an energy functional for wave or Klein-Gordon components on hyperboloids,
we first need to recall some notation from \cite{PLF-YM-book} concerning the hyperboloidal foliation method. We adopt the signature $(-, +, +, +)$ in the $(3+1)$-- dimensional Minkowski spacetime, and we denote the point $(t, x) = (x^0, x^1, x^2, x^3)$ in Cartesion coordinates, with its spatial radius $r := | x | = \sqrt{(x^1)^2 + (x^2)^2 + (x^3)^2}$. We write $\del_\alpha$ (for $\alpha=0, 1, 2, 3$) for partial derivatives and 
\be
L_a := x^a \del_t + t \del_a, \qquad a= 1, 2, 3
\ee
represent the Lorentz boosts. Throughout the paper, we consider functions defined in the interior of the future light cone $\Kcal:= \{(t, x): r< t-1 \}$, with vertex $(1, 0, 0, 0)$. We consider hyperboloidal hypersurfaces $\Hcal_s:= \{(t, x): t^2 - r^2 = s^2 \}$ with $s>1$. Note that for all points on $\Kcal \cap \Hcal_s$ with $s > 1$ it holds 
$$
s \leq t \leq s^2.
$$
Also $\Kcal_{[s_0, s_1]} := \{(t, x): s_0^2 \leq t^2- r^2 \leq s_1^2; r<t-1 \}$ is used to denote subsets of $\Kcal$ limited by two hyperboloids $\Hcal_{s_0}$ and $\Hcal_{s_1}$ with $s_0 \leq s_1$.

The semi-hyperboloidal frame is defined by
\bel{eq:semi-hyper}
\underdel_0:= \del_t, \qquad \underdel_a:= {L_a \over t} = {x^a\over t}\del_t+ \del_a.
\ee
Note that the vectors $\underdel_a$ generate the tangent space to the hyperboloids. We also introduce the vector field $\underdel_\perp:= \del_t+ (x^a / t)\del_a$, which is orthogonal to the hyperboloids.

For the semi-hyperboloidal frame above, the dual frame is given by $\underline{\theta}^0:= dt- (x^a / t)dx^a$ and $\underline{\theta}^a:= dx^a$. The (dual) semi-hyperboloidal frame and the (dual) natural Cartesian frame are connected by the relation
\bel{semi-hyper-Cts}
\underdel_\alpha= \Phi_\alpha^{\alpha'}\del_{\alpha'}, \quad \del_\alpha= \Psi_\alpha^{\alpha'}\underdel_{\alpha'}, \quad \underline{\theta}^\alpha= \Psi^\alpha_{\alpha'}dx^{\alpha'}, \quad dx^\alpha= \Phi^\alpha_{\alpha'}\underline{\theta}^{\alpha'},
\ee
where the transition matrix ($\Phi^\beta_\alpha$) and its inverse ($\Psi^\beta_\alpha$) are given by
\be
(\Phi_\alpha^{ \beta})=
\begin{pmatrix}
1 & 0 &   0 &  0   \\
{x^1 / t} & 1  & 0   &  0  \\
{x^2 / t} &  0  &  1  &  0   \\
{x^3 / t} &  0 & 0   & 1
\end{pmatrix}
\ee
and 
\be
(\Psi_\alpha^{ \beta})=
\begin{pmatrix}
1 & 0 &   0 &  0   \\
-{x^1 / t} & 1  & 0   &  0  \\
-{x^2 / t} &  0  &  1  &  0   \\
-{x^3 / t} &  0 & 0   & 1
\end{pmatrix}.
\ee


\subsection{Energy estimates on hyperboloids}

Following \cite{PLF-YM-cmp}, we first introduce the energy $E_m$, in the Minkowski background, for a function $\phi$ defined on a hyperboloid $\Hcal_s$: 
\bel{eq:2energy} 
\aligned
E_m(s, \phi)
&:=
 \int_{\Hcal_s} \Big( \big(\del_t \phi \big)^2+ \sum_a \big(\del_a \phi \big)^2+ 2 (x^a/t) \del_t \phi \del_a \phi + m^2 \phi ^2 \Big) \, dx
\\
               &= \int_{\Hcal_s} \Big( \big( (s/t)\del_t \phi \big)^2+ \sum_a \big(\underdel_a \phi \big)^2+ m^2 \phi^2 \Big) \, dx
                \\
               &= \int_{\Hcal_s} \Big( \big( \underdel_\perp \phi \big)^2+ \sum_a \big( (s/t)\del_a \phi \big)^2+ \sum_{a<b} \big( t^{-1}\Omega_{ab} \phi \big)^2+ m^2 \phi^2 \Big) \, dx,
 \endaligned
 \ee
in which $\Omega_{ab}:= x^a\del_b- x^b\del_a$ the rotational vector field, $\underdel_{\perp}:= \del_t+ (x^a / t) \del_a$ the orthogonal vector field, and we denote $E(s, \phi):= E_0(s, \phi)$ for simplicity.
In the above, the integral in $L^1(\Hcal_s)$ is defined from the standard (flat) metric in $\RR^3$, i.e.
\bel{flat-int}
\|\phi \|_{L^1_f(\Hcal_s)}
:=\int_{\Hcal_s}|\phi | \, dx 
=\int_{\RR^3} \big|\phi(\sqrt{s^2+r^2}, x) \big| \, dx.
\ee

Next, we adapt the energy estimates to our situation.

\begin{proposition}[Energy estimate for wave equation]
For all $s \geq 2$, it holds that
\bel{eq:w-EE} 
E(s, u)^{1/2}
\leq 
E(2, u)^{1/2}
+ \int_2^s \| \Box u \|_{L^2_f(\Hcal_{s'})} \, ds'
\ee
for every sufficiently regular function $u$, which is defined and supported in the region $\Kcal_{[2, s]}$.
\end{proposition}

For the proof, one refers to \cite{PLF-YM-cmp}.

\begin{proposition}[Energy estimate for Klein-Gordon equation with varying mass]
Let $v$ be a solution to the Klein-Gordon equation with mass 1
\bel{eq:KG-1} 
- \Box v +  v = u v + f,
\ee
which can also be regarded as a Klein-Gordon equation with varying mass $1 - u$
\bel{eq:KG-vary}
- \Box v + (1 - u) v = f,
\ee
defined and supported in the region $\Kcal_{[2, s]}$, and $u$ is a sufficiently regular function defined and supported in the same region $\Kcal_{[2, s]}$, which is assumed to be small
\be 
| u | \leq {1\over 10}.
\ee
Then the energy on the hyperboloid $\Hcal_s$ can be controlled by both 
\bel{eq:est-KG-1}
E_1(s, v)^{1/2}
\leq 
E_1(2, v)^{1/2}
+ \int_2^s \Big( \| u v\|_{L^2_f(\Hcal_{s'})} + \|f \|_{L^2_f(\Hcal_{s'})} \Big) ds',
\ee
and 
\bel{eq:est-KG-vary} 
E_1(s, v)^{1/2}
\leq 
2 E_1(2, v)^{1/2}
+2 \int_2^s \Big( \| (s' / t) \del_t u v\|_{L^2_f(\Hcal_{s'})} + \|f \|_{L^2_f(\Hcal_{s'})} \Big) ds'.
\ee
\end{proposition}

The energy estimate \eqref{eq:est-KG-vary} is better than \eqref{eq:est-KG-1} in the cases where $\del_t u$ decays faster than $u$, which is the case when $u$ is a solution to some wave equation.

\begin{proof}
The proof of the energy estimate \eqref{eq:est-KG-1} is standard and we omit it.
In order to prove the energy estimate \eqref{eq:est-KG-vary}, we first test the equation \eqref{eq:KG-vary} by the multiplier $\del_t v$ and write the resulting equation in the following favorable form
\bel{eq:div-form}
{1\over 2} \del_t \big( (\del_t v)^2 + \sum_a (\del_a v)^2 + (1 - u) v^2 \big)
+ \sum_a \del_a \big( - \del_t v \del_a v \big)
=
- {1 \over 2} v^2 \del_t u 
+ \del_t v f.
\ee
We then integrate the identity \eqref{eq:div-form} over the region $\Kcal_{[2, s]}$ and do integration by parts to arrive at
\bel{eq:together} 
\aligned
&\quad E_{\sqrt{1 - u}}(s, v)^{1/2} {d \over ds} E_{\sqrt{1 - u}}(s, v)^{1/2}
\\
&=
\int_{\Hcal_s} (s / t) \big( - {1 \over 2} v^2 \del_t u  + \del_t v f \big) \, dx
\\
&\leq
\| (s / t) \del_t u v \|_{L^2_f(\Hcal_s)} \|v \|_{L^2_f(\Hcal_s)} + \| f \|_{L^2_f(\Hcal_s)} \| (s / t) \del_t v \|_{L^2_f(\Hcal_s)}.
\endaligned
\ee
Next by recalling the assumption that $|u| \leq 1/10$, we have
$$
\aligned
{9 \over 10} E_1(s, v)^{1/2}
\leq
E_{\sqrt{1 - u}}(s, v)^{1/2}
\leq
{11 \over 10 }E_1(s, v)^{1/2},
\endaligned
$$
which together with \eqref{eq:together} leads to 
$$
E_{\sqrt{1 - u}}(s, v)^{1/2}
\leq 
E_{\sqrt{1 - u}}(2, v)^{1/2}
+ {11 \over 10} \int_2^s \Big( \| v \del_t u \|_{L^2(\Hcal_{s'})} + \|f \|_{L^2(\Hcal_{s'})} \Big) ds',
$$
and finally \eqref{eq:est-KG-vary}.
\end{proof}


\subsection{Conformal-type energy estimates on hyperboloids}
\label{subsc:conf}
We now introduce a conformal-type energy which is adapted to the hyperboloidal foliation setting, which is due to Ma and Huang in \cite{YM-HH}. This lemma will be key to a robust estimate of the $L^2$-type norm for the wave component $u$.

\begin{lemma}
Define the conformal-type energy of a sufficiently regular function $u$, which is supported in the region $\Kcal = \{(t, x): |x|< t - 1\}$, by
\be 
E_{con} (u, s)
:=
\int_{\Hcal_s} \Big( \sum_a \big( s \underdel_a u \big)^2 + \big( K u + 2 u \big)^2 \Big) \, dx,
\ee
in which we used the notation of the weighted inverted time translation
$$
K u 
:= \big( s \del_s + 2 x^a \underdel_a \big) u.
$$
Then it holds
\bel{eq:con-estimate} 
E_{con} (u, s)^{1/2}
\leq 
E_{con} (u, s_0)^{1/2}
+
2 \int_{s_0}^s s' \| \Box u \|_{L^2_f(\Hcal_{s'})} \, ds',
\ee
with moreover
\bel{eq:l2type-wave} 
\big\| (s / r) u \big\|_{L^2_f(\Hcal_s)}
\leq
E_{con} (u, s)^{1/2}.
\ee
\end{lemma}


\subsection{Sobolev-type and Hardy-type inequality}

We first state a Sobolev-type inequality adapted to the hyperboloids, which is of vital importance for proving sup-norm estimates for both wave and Klein-Gordon components. For the proof, one refers to either \cite{PLF-YM-book} or \cite{PLF-YM-cmp} for details.

\begin{lemma} \label{lem:sobolev}
For all sufficient smooth functions $u= u(t, x)$ supported in $\{(t, x): |x|< t - 1\}$ and for all  $s \geq 2$, one has 
\bel{eq:Sobolev2}
\sup_{\Hcal_s} \big| t^{3/2} u(t, x) \big|  \lesssim \sum_{| J |\leq 2} \| L^J u \|_{L^2_f(\Hcal_s)},
\ee
in which the symbol $L$ denotes the Lorentz boosts and $J$ is a multi-index. We will also frequently make use of the following identity which follows from \eqref{eq:Sobolev2} and standard commutator estimates:
\be
\sup_{\Hcal_s} \big| s \hskip0.03cm t^{1/2} u(t, x) \big|  \lesssim \sum_{| J |\leq 2} \| (s/t) L^J u \|_{L^2_f(\Hcal_s)},
\ee
\end{lemma}

In order to control the $L^2$--type of norm for the wave component $u$, we need the following Hardy-type inequality on the hyperboloidal foliation, see \cite{PLF-YM-book} for instance.
\begin{lemma}
Assume the function $u$ is defined and supported in the region $\{(t, x): |x|< t - 1\}$ and is sufficiently regular, then for all $s\geq 2$, one has 
\bel{eq:Hardy} 
\| r^{-1} u \|_{L^2_f(\Hcal_s)}
\lesssim
\sum_{a} \| \underdel_a u \|_{L^2_f(\Hcal_s)}.
\ee
\end{lemma}


\section{Estimates for commutators and null forms}
\label{sec:ECNF}

\subsection{Commutator estimates}

We state the estimates for the commutators, which are proven in \cite{PLF-YM-book} and \cite{PLF-YM-cmp}.

\begin{lemma} \label{lem:est-comm}
Assume a function $u$ defined in the region $\mathcal{K}$ is regular enough, then with the generic constant $C(|I|, |J|)$, we have
\begin{align}
\label{eq:est-cmt1}
& \big| [\del^I L^J, \del_\alpha] u \big| 
\leq 
C(| I |, |J|) \sum_{|J'|<|J|, \beta} \big|\del_\beta \del^I L^{J'} u \big|,
\\
\label{eq:est-cmt2}
& \big| [\del^I L^J, \underdel_a] u \big| 
\leq 
C(| I |, |J|) \Big( \sum_{| I' |<| I |, | J' |< | J |, b} \big|\underdel_b \del^{I'} L^{J'} u \big| + t^{-1} \sum_{| I' |\leq | I |, |J'|\leq |J|} \big| \del^{I'} L^{J'} u \big| \Big),
\\
\label{eq:est-cmt3}
& \big| [\del^I L^J, \underdel_\alpha] u \big| 
\leq 
C(| I |, |J|) \Big( \sum_{| I' |<| I |, | J' |< | J |, \beta} \big|\del_\beta \del^{I'} L^{J'} u \big| + t^{-1} \sum_{| I' |\leq | I |, |J'|\leq |J|, \beta} \big| \del_\beta \del^{I'} L^{J'} u \big| \Big),
\\
\label{eq:est-cmt4}
& \big| [\del^I L^J, \del_\alpha \del_\beta] u \big| 
\leq 
C(| I |, |J|) \sum_{| I' |\leq  | I |, |J'|<|J|, \gamma, \gamma '} \big| \del_\gamma \del_{\gamma '} \del^{I'} L^{J'} u \big|,
\\
\label{eq:est-cmt6}
& \big| \del^I L^J ((s/t) \del_\alpha u) \big| 
\leq 
|(s/t) \del_\alpha \del^I L^J u| + C(| I |, |J|) \sum_{| I' |\leq  | I |, |J'|\leq |J|, \beta } \big|(s/t) \del_\beta \del^{I'} L^{J'} u \big|.
\end{align}
Recall here that Greek indices $\alpha, \beta \in \{0,1,2,3\}$ and Latin indices $a,b \in \{1,2,3\}$. 
\end{lemma}


\subsection{Null form estimates}

\begin{lemma}
\label{lem:null}
For the quadratic null term $\del^\alpha u \del_\alpha v$ 
with sufficiently regular functions $u$ and $v$, one has 
\bel{eq:est-null1}
\aligned
\big| \del^I L^J ( \del^\alpha u \del_\alpha v) \big| 
&\lesssim 
\sum_{\substack{| I_1 | + | I_2 | \leq | I |, \\ |J_1| + |J_2| \leq |J|, \\ a, \beta}} \Big( \big| \del^{I_1} L^{J_1} \underdel_a u \del^{I_2} L^{J_2} \underdel_\beta v \big| 
+ \big| \del^{I_1} L^{J_1} \underdel_\beta u \del^{I_2} L^{J_2} \underdel_a v \big| \Big)
\\
& + (s/t)^2 \sum_{\substack{| I_1 | + | I_2 | \leq | I |,\\ |J_1| + |J_2| \leq |J| }} \big|\del^{I_1} L^{J_1} \del_t u \del^{I_2} L^{J_2} \del_t v \big|.
\endaligned
\ee
\end{lemma}
One refers to \cite{PLF-YM-book} for the proof.


\section{Tools for pointwise estimates for wave and Klein-Gordon components}
\label{sec:P-WKG}

\subsection{Sup-norm estimates for wave components}

We recall the following lemma from \cite{PLF-YM-cmp}, which is essential in proving the sup-norm bound for wave components. An alternative proof of Lemma \ref{lem:supwave} is also found in \cite{Alinhac}.

\begin{lemma}[Pointwise estimates for wave components]
\label{lem:supwave}
Suppose $u$ is a spatially compactly supported solution to the wave equation
\be 
\aligned
- \Box u &= f,
\\
u(t_0, x) &= \del_t u(t_0, x) =0,
\endaligned
\ee
with $f$ spatially compactly supported and satisfying
\be 
| f | \leq C_f t^{- 2 - \nu} (t - r)^{-1 + \mu},
\ee
for $0<\mu \leq 1/2$ and $0<\nu \leq 1/2 $. Then we have 
\be 
| u(t, x) | \lesssim {C_f \over \nu \mu} (t - r)^{\mu - \nu} t^{-1},
\ee
where $C_f$ is some constant. 
\end{lemma}


\subsection{Sup-norm estimates for Klein-Gordon components}

Following the pointwise estimates for Klein-Gordon components in the hyperboloidal foliation setting, which were first introduced in \cite{PLF-YM-cmp}, we adapt it to our case where the mass of the Klein-Gordon field varies.

\begin{proposition}[Pointwise estimates for Klein-Gordon components with varying mass]
\label{lem:supKG}
Assume $v$ is a sufficiently regular and spatially compactly supported solution to the Klein-Gordon equation
\bel{eq:KG-vary2} 
\aligned
- \Box v + (1 - u) v 
&= f,
\\
v|_{\Hcal_{2}} 
=v_0, 
\quad 
& \del_t v|_{\Hcal_{2}} 
= v_1,
\endaligned
\ee
with the assumption $| u | \leq 1/10$, 
then one has 
\bel{eq:supKG1}
s^{3/2} \big| v(t, x) | + {(s/t)}^{-1} s^{3/2} | \underdel_\perp v(t, x) \big| 
\lesssim V(t, x),
\ee
with 
\begin{eqnarray}
V(t, x):=
\left\{
\begin{array}{lll}
& e^{\int_{s_0}^s | {d\over d\lambda} u(\lambda t/s, \lambda x/s)| \, d\lambda} \Big(\|v_0\|_{L^{\infty(\Hcal_2)}} + \| v_1 \|_{L^{\infty(\Hcal_2)}} + F(s) \Big),  &\quad {r / t} \leq 3/5,
\\
& e^{\int_{s_0}^s | {d\over d\lambda} u(\lambda t/s, \lambda x/s)| \, d\lambda}  F(s),    & 3/5 \leq  {r / t} \leq 1,
\end{array}
\right.
\end{eqnarray}
and 
\begin{eqnarray}
s_0:=
\left\{
\begin{array}{lll}
2, & \quad r / t \leq 3/5,
\\
\sqrt{t+r \over t-r}, & 3/5 \leq  {r / t} \leq 1,
\end{array}
\right.
\end{eqnarray}
 and
\bel{eq:odesource}
F(s) 
:= \int_{s_0}^s  \Big| R[v](\lambda t/s, \lambda x/s) + \lambda^{3/2} f(\lambda t/s, \lambda x/s) \Big| \, d\lambda,
\ee
where
\be 
R[v] 
:= 
s^{3/2} \sum \underdel_a \underdel_a v + {x^a x^b \over s^{1/2}} \underdel_a \underdel_b v + {3 \over 4 s^{1/2}} v + {3 x^a \over s^{1/2}} \underdel_a v.
\ee
\end{proposition}

The proof of Proposition \ref{lem:supKG} is based on the decomposition result in Lemma~\ref{lem:decompose} and an estimate of ODE in Lemma~\ref{lem:ODE}, both stated below. We refer to \cite{PLF-YM-cmp} for the detailed proofs, but give a simpler proof of Lemma~\ref{lem:ODE} below, which provides a neater expression of the estimate for the ODE.  

\begin{lemma}
\label{lem:decompose}
Assume $v$ is a sufficiently regular solution to the Klein-Gordon equation \eqref{eq:KG-vary2}, and let 
$$
w_{t, x}(\lambda)
:=
\lambda^{3/2} v(\lambda t/s, \lambda x/s),
\quad
(t, x) \in \Kcal,
$$
then the following second-order ODE with respect to $\lambda$ holds
\be 
{d^2 \over d \lambda^2} w_{t, x}(\lambda) + \big(1 - u(\lambda t/s, \lambda x/s) \big) w_{t, x}(\lambda)
=
\big( R[v] + s^{3/2} f \big) (\lambda t/s, \lambda x/s).
\ee
\end{lemma}

\begin{lemma}
\label{lem:ODE}
Consider the second-order ODE
\bel{eq:ode1} 
\aligned
& z''(\lambda) + \big( 1 - G(\lambda) \big) z(\lambda) =  k(\lambda),
\\
& z(s_0) = z_0, \quad z'(s_0) =  z_1, \quad  | G(\lambda) | \leq 1/10,
\endaligned
\ee
in which $k$ is assumed to be integrable, then we have the following pointwise estimate
\be 
\big( (z')^2(s) + (1 - G(s)) z^2(s) \big)^{1/2} \lesssim e^{\int_{s_0}^s | G'(\lambda)| \, d\lambda} \Big( \big( (z')^2(0) + z^2(0) \big)^{1/2} + \int_{s_0}^s |k(\lambda)| \, d\lambda \Big).
\ee
\end{lemma}
\begin{proof}
We set $Y(\lambda) = \big((z')^2(\lambda) + (1 - G(\lambda)) z^2(\lambda) \big)^{1/2}$, and then by multiplying $z'(\lambda)$ in \eqref{eq:ode1}, we get
\be
\aligned 
{d \over d \lambda} Y^2(\lambda) &= z'(\lambda) k(\lambda) - G'(\lambda) z^2(\lambda)
\\
&\leq Y(\lambda) \big( |k(\lambda)| + | G' | Y(\lambda) \big).
\endaligned
\ee
In order to proceed, we divide $Y(\lambda)$ in the above inequality and, integrate to get
\be 
Y(s) \leq Y(s_0) + \int_{s_0}^s \big( |k(\lambda)| + | G' | Y(\lambda) \big) \, d\lambda.
\ee
Finally, we apply Gronwall-type inequality from Lemma \ref{lem:Gronwall} to end the proof.
\end{proof}

We have used the following standard Gronwall inequality.

\begin{lemma}
\label{lem:Gronwall}
Let u(t) be continuous and nonnegative in $[0, T]$, and satisfy
\be 
u(t) \leq A + \int_0^t \Big( a(s) u(s) + b(s) \Big) \, ds,
\ee
where $a(t)$ and $b(t)$ are nonnegative integrable functions in $[0, T]$ and A is nonnegative constant. Then it holds
\be 
u(t) \leq \Big( A + \int_0^t b(s) \, ds \Big) e^{\int_0^t a(s) \, ds}, \quad t \in [0, T].
\ee
\end{lemma}


\section{Bootstrap method}
\label{sec:BM}

\subsection{Bootstrap assumption}
\label{subsect:BA}

We assume that the following bootstrap assumptions hold in the interval $[2, s_1)$
\begin{subequations}
\label{eq:boots-assum}
\begin{align}
E(s, \del^I L^J u)^{1/2} \label{boots-assum1}
&\leq 
C_1 \eps,
&\quad   |I| + |J| \leq N - 1,
\\
E(s, \del^I u)^{1/2} \label{boots-assum2}
&\leq 
C_1 \eps s^\delta,
&\quad   |I| = N,
\\
E(s, \del^I L^J u)^{1/2} \label{boots-assum3}
&\leq 
C_1 \eps s^{|J| \delta},
&\quad   |I| + |J| \leq N, |J| \geq 1
\\
E_1(s, \del^I L^J v)^{1/2} \label{boots-assum4}
&\leq 
C_1 \eps s^{|J| \delta},
&\quad   |I| + |J| \leq N,
\\
\| (s/t) \del^I L^J u \|_{L^2_f(\Hcal_s)} \label{boots-assum5}
&\leq
C_1 \eps s^{1/2 + |J| \delta},
&\quad   |I| + |J| \leq N,
\\
|\del^I L^J u | \label{boots-assum6}
&\leq 
C_1 \eps t^{-1} s^{|J| \delta},
&\quad   |I| + |J| \leq N - 4,
\\
|\del^I L^J v | \label{boots-assum7}
&\leq
(C_1 \eps)^{1/2} t^{-3/2} s^{|J| \delta},
&\quad   |I| + |J| \leq N - 4,
\end{align}
\end{subequations}
in which $C_1$ is some big constant which is fixed once and for all and will be chosen to satisfy $C_1 \eps \ll 1$, $\delta$ is some fixed small constant, such that $0 < \delta \ll 1/N$, and $s_1$ is defined by
$$
s_1
:=
\sup \{ s: \eqref{eq:boots-assum} ~hold \}.
$$
We recall that the fact $s_1 > 2$ follows from the local existence result, which is classical, see for example \cite[Section 11]{PLF-YM-book}.
And importantly, we note that $C_1$ and $\delta$ are independent of $s_1$. 

In order to prove the stability result stated in Theorem~\ref{thm:wKG}, it suffices to demonstrate the refined energy bounds below
\bel{eq:boots-refine}
\aligned
E(s, \del^I L^J u)^{1/2}
&\leq 
{1\over 2} C_1 \eps,
&\quad   |I| + |J| \leq N - 1,
\\
E(s, \del^I u)^{1/2}
&\leq 
{1\over 2} C_1 \eps s^\delta,
&\quad   |I| = N,
\\
E(s, \del^I L^J u)^{1/2}
&\leq 
{1\over 2} C_1 \eps s^{|J| \delta},
&\quad   |I| + |J| \leq N, |J| \geq 1,
\\
E_1(s, \del^I L^J v)^{1/2}
&\leq 
{1\over 2} C_1 \eps s^{|J| \delta},
&\quad   |I| + |J| \leq N,
\\
\| (s/t) \del^I L^J u \|_{L^2_f(\Hcal_s)}
&\leq
{1\over 2} C_1 \eps s^{1/2 + |J| \delta},
&\quad   |I| + |J| \leq N,
\\
|\del^I L^J u |
&\leq 
{1\over 2} C_1 \eps t^{-1} s^{|J| \delta},
&\quad   |I| + |J| \leq N - 4,
\\
|\del^I L^J v |
&\leq
{1\over 2} (C_1 \eps)^{1/2} t^{-3/2} s^{|J| \delta},
&\quad   |I| + |J| \leq N - 4.
\endaligned
\ee
Note that the bounds in \eqref{eq:boots-refine} indicate that $s_1$ cannot be of finite value, which thus completes the proof of a global-in-time solution stated in the main Theorem \ref{thm:wKG}.


\subsection{Direct estimates}
\label{subsec:DE}

Direct consequences of \eqref{boots-assum1} and \eqref{boots-assum4} are the following:
\bel{eq:D1}
\aligned
|\del^I L^J \del u | + |\del \del^I L^J u |
&\lesssim
C_1 \eps t^{-1/2} s^{-1},
\quad
& |I| + |J| \leq N - 3,
\\
|\del^I L^J v |
&\lesssim
C_1 \eps t^{-3/2} s^{(|J| + 2) \delta},
\quad
& |I| + |J| \leq N - 2.
\endaligned
\ee
These follow from the Sobolev--type inequality of Lemma \ref{lem:sobolev}  and estimates for commutators in Lemma \ref{lem:est-comm}.

Assumptions \eqref{boots-assum1}--\eqref{boots-assum3} also imply the following $L^2$--type estimates
\bel{eq:D2}
\aligned
\| (s/t) \del^I L^J \del u \|_{L^2_f(\Hcal_s)} + \| (s/t) \del \del^I L^J u \|_{L^2_f(\Hcal_s)}
&\lesssim
C_1 \eps,
\quad
& |I| + |J| \leq N - 1,
\\
\| (s/t) \del^I \del u \|_{L^2_f(\Hcal_s)} + \| (s/t) \del \del^I u \|_{L^2_f(\Hcal_s)}
&\lesssim
C_1 \eps s^\delta,
\quad
& |I| = N,
\\
\| (s/t) \del^I L^J \del u \|_{L^2_f(\Hcal_s)} + \| (s/t) \del \del^I L^J u \|_{L^2_f(\Hcal_s)}
&\lesssim
C_1 \eps s^{|J| \delta},
\quad
& |I| + |J| = N, |J| \geq 1.
\endaligned
\ee


\section{Refined estimates for the Klein-Gordon component}
\label{sec:refineKG}

\subsection{Refined energy estimates for $v$}

We show here the refined estimates for the Klein-Gordon component, and we will see that the most difficult part is to get the refined ones for $\del^I v$. The difficulty comes from the integral of 
$$
\int_2^s s'^{-1} \, ds'
$$
diverges, but we can circumvent it by moving the nonlinear term $u v$ in the Klein-Gordon equation in \eqref{eq:model} to the left hand side and then regarding the mass of $v$ as the varying one $1 - u$.

\begin{lemma} \label{lem:comm-vary-mass}
By utilising the notation of commutators $[A, B] u := A (B u) - B (A u)$, we have
\bel{eq:comm-01} 
\big\| [1 - u, \del^I L^J ] v \big\|_{L^2_f(\Hcal_s)}
\lesssim
(C_1 \eps)^{3/2} s^{-1 + |J| \delta},
\quad
|I| + |J| \leq N,
\ee
and furthermore, we have
\bel{eq:comm-02} 
\big\| [1 - u, \del^I ] v \big\|_{L^2_f(\Hcal_s)}
\lesssim
(C_1 \eps)^{3/2} s^{-3/2},
\quad
|I| \leq N.
\ee
\end{lemma}
\begin{proof}
We first prove \eqref{eq:comm-01} and note the expansion of the commutator
$$
[1 - u, \del^I L^J ] v
=
 \sum_{\substack{I_1 + I_2 = I, J_1 + J_2 = J\\ |I_1| + |J_1| \geq 1}}
\del^{I_1} L^{J_1} u \del^{I_2} L^{J_2} v.
$$
For the case of $|J| \geq 1$, we conduct the following
$$
\aligned
\big\| [1 - u, \del^I L^J ] v \big\|_{L^2_f(\Hcal_s)}
&\lesssim
\sum_{\substack{I_1 + I_2 = I, J_1 + J_2 = J\\ |I_1| + |J_1| \geq |I_2| + |J_2|}} \| (s/t) \del^{I_1} L^{J_1} u \|_{L^2_f(\Hcal_s)} \| (t/s) \del^{I_2} L^{J_2} v \|_{L^\infty(\Hcal_s)}
\\
&+
\sum_{\substack{I_1 + I_2 = I, J_1 + J_2 = J\\ 1 \leq |I_1| + |J_1| \leq |I_2| + |J_2|}} \| \del^{I_1} L^{J_1} u \|_{L^\infty(\Hcal_s)} \| \del^{I_2} L^{J_2} v \|_{L^2_f(\Hcal_s)},
\endaligned
$$
and the $L^2$--type estimates for $u$ in \eqref{eq:boots-assum} verifies
$$
\big\| [1 - u, \del^I L^J ] v \big\|_{L^2_f(\Hcal_s)}
\lesssim
\sum_{J_1 + J_2 = J} C_1 \eps s^{1/2 + |J_1| \delta} (C_1 \eps)^{1/2} t^{-1/2} s^{-1 + |J_2| \delta}
+ 
C_1 \eps t^{-1} s^{|J_1| \delta} C_1 \eps s^{|J_2| \delta},
$$
which leads to \eqref{eq:comm-01}.

For the proof of \eqref{eq:comm-02} with $|I| \geq 1$, we proceed in the same way 
$$
[1 - u, \del^I ] v
=
 \sum_{\substack{I_1 + I_2 = I \\ |I_1| \geq 1}}
\del^{I_1} u \del^{I_2} v.
$$
We note that there exists at least one derivative hitting on the wave component $u$,
and recall the fact that 
$$
\aligned
\big\| (s/t) \del^{I_1} u \big\|_{L^2_f(\Hcal_s)}
&\lesssim
C_1 \eps,
\quad
&1 \leq |I_1| \leq N ,
\\
\big\| (s/t) \del^{I_1} u \big\|_{L^\infty(\Hcal_s)}
&\lesssim
C_1 \eps t^{-3/2},
&\quad
1 \leq |I_1| \leq N - 4.
\endaligned
$$
Then we have
$$
\aligned
\big\| [1 - u, \del^I] v \big\|_{L^2_f(\Hcal_s)}
&\lesssim
\sum_{\substack{I_1 + I_2 = I\\ |I_1| \geq 1, |I_1| \geq |I_2|}} \| (s/t) \del^{I_1}  u \|_{L^2_f(\Hcal_s)} \| (t/s) \del^{I_2} v \|_{L^\infty(\Hcal_s)}
\\
&+
\sum_{\substack{I_1 + I_2 = I \\ 1 \leq |I_1|  \leq |I_2|}} \| \del^{I_1} u \|_{L^\infty(\Hcal_s)} \| \del^{I_2} v \|_{L^2_f(\Hcal_s)}
\\
&\lesssim (C_1 \eps)^{3/2} s^{-1} t^{-1/2} \lesssim (C_1 \eps)^{3/2} s^{-3/2}.
\endaligned
$$

Finally, since $[1 - u, \del^I L^J] = [1 - u, \del^I L^J] = 0$ when $|I| = |J| = 0$, the proof is hence complete.

\end{proof}

\begin{proposition}[Refined energy estimates for $v$]
Consider the Klein-Gordon equation in \eqref{eq:model} and assume the bounds in \eqref{eq:boots-assum} hold, then we have the following refined ones
\be 
E_1(s, \del^I L^J v)^{1/2}
\lesssim  
\eps + (C_1 \eps)^{3/2} s^{|J| \delta},
\quad
|I| + |J| \leq N.
\ee
\end{proposition}
\begin{proof}

We first act $\del^I L^J$ on the Klein-Gordon equation in \eqref{eq:model} to get
$$
- \Box \del^I L^J v + (1 - u) \del^I L^J v
=
 \sum_{\substack{I_1 + I_2 = I, J_1 + J_2 = J\\ |I_1| + |J_1| \geq 1}}
\del^{I_1} L^{J_1} u \del^{I_2} L^{J_2} v.
$$
We then apply the energy estimate \eqref{eq:est-KG-vary} for Klein-Gordon equations with varying masses and use Lemma \ref{lem:comm-vary-mass} to show
\be 
\aligned
&\quad
E_1(s, \del^I L^J  v)^{1/2}
\\
&\leq
2 E(2, \del^I L^J  v)^{1/2}
	+ 2 \int_2^s \Big( \| (s' / t) \del_t u \del^I L^J  v \|_{L^2_f(\Hcal_{s'})} 
	+  \big\| [1 - u, \del^I L^J ] v \big\|_{L^2_f(\Hcal_s)}  ds'
\\
&\lesssim
\eps 
+ \int_2^s \Big( \| \del_t u \|_{L^\infty(\Hcal_{s'})} \| \del^I L^J v \|_{L^2_f(\Hcal_{s'})} 
+	\sum_{\substack{I_1 + I_2 = I, J_1 + J_2 = J\\ |I_1| + |J_1| \geq 1}} \|
\del^{I_1} L^{J_1} u \del^{I_2} L^{J_2} v\|_{L^2_f(\Hcal_{s'})} \Big) \,\, ds'.
\endaligned
\ee
Successively, in the case of $|J| \geq 1$, it is true that
\be 
E_1(s, \del^I L^J  v)^{1/2}
\lesssim
\eps + (C_1 \eps)^{3/2} \int_2^s s'^{-1 + |J| \delta}\, ds'
\lesssim 
\eps + (C_1 \eps)^{3/2} s^{|J| \delta},
\ee
while in the case of $|J| = 0$, better estimates on $\del^{I_1} u$ with $|I_1| \geq 1$ enable us to obtain
\be 
E_1(s, \del^I L^J  v)^{1/2}
\lesssim
\eps + (C_1 \eps)^{3/2} \int_2^s s'^{-3/2 + \delta}\, ds'
\lesssim 
\eps + (C_1 \eps)^{3/2},
\ee
which finishes the proof.
\end{proof}


\subsection{Refined pointwise estimates for $v$}

We now prove the refined sup-norm bounds for the Klein-Gordon component $v$,and  we first prepare some lemmas which will be of help.

\begin{lemma}\label{lem6.3}
The solution $u$ to our wave equation satisfies
\be 
e^{\int_{s_0}^s | {d\over d\lambda} u(\lambda t/s, \lambda x/s)| \, d\lambda} 
\lesssim 1.
\ee
\end{lemma}
\begin{proof}
We observe that 
$$
 {d\over d\lambda} u(\lambda t/s, \lambda x/s)
 =
 (t/s) \underdel_\perp u(\lambda t/s, \lambda x/s),
$$
and, on the other hand, we have
$$
 \underdel_\perp u (t, x)
 =
{s^2 \over t^2} \del_t u (t, x) + {x^a \over t^2} L_a u (t, x).
$$
Hence by recalling the pointwise bootstrap \eqref{boots-assum6} of $u$ that
$$
|L_a u(t, x) |
\leq
C_1 \eps t^{-1} s^{\delta},
$$
we find
$$
\big| (t / s) \underdel_\perp u(t, x) \big|
\lesssim
C_1 \eps s^{-3/2}.
$$

This implies that
$$
\Big|  {d\over d\lambda} u(\lambda t/s, \lambda x/s) \Big|
 \lesssim C_1 \eps \lambda^{-3/2},
$$
and hence the completeness of the proof. 
\end{proof}

\begin{lemma}\label{lem:6.4}
We have the estimate for $R[\del^I L^J v]$ in the region $\Kcal_{[2, s_1)}$ that
\bel{eq:est-R} 
\big| R[\del^I L^J v](\lambda t/s, \lambda x/s) \big|
\lesssim
C_1 \eps (s / t)^{3/2} \lambda^{-3/2 + N \delta},
\quad
|I| + |J| \leq N - 4.
\ee
\end{lemma}
The proof can be found in \cite{PLF-YM-cmp}.

One last ingredient is the commutator estimate stated below.

\begin{lemma}\label{lem:6.5}
The following estimates for the the commutator are valid
\bel{eq:comm-001}
\big| \big([1 - u, \del^I L^J ] v \big)(\lambda t/s, \lambda x/s) \big|
\lesssim
(C_1 \eps)^{3/2} (s/t)^{5/2} \lambda^{-5/2 + |J| \delta},
\quad
|I| + |J| \leq N - 4,
\ee
moreover, in the case of $|J| = 0$, one has
\bel{eq:comm-002} 
\big| \big([1 - u, \del^I ] v \big)(\lambda t/s, \lambda x/s) \big|
\lesssim
(C_1 \eps)^{3/2} (s/t)^{2} \lambda^{-3}, 
\quad
|I| \leq N - 4.
\ee
\end{lemma}
\begin{proof}
In order to show \eqref{eq:comm-001}, first recall the expansion of the commutator
$$
[1 - u, \del^I L^J ] v
=
\sum_{\substack{I_1 + I_2 = I, J_1 + J_2 = J \\ |I_1| + |J_1| \geq 1}}
\del^{I_1} L^{J_1} u \del^{I_2} L^{J_2} v.
$$
Next recall the pointwise estimates in \eqref{eq:boots-assum} and they give
$$
\aligned
|\big( [1 - u, \del^I L^J ] v \big) (t, x)|
&\lesssim
\sum_{J_1 + J_2 = J} C_1 \eps t^{-1} s^{|J_1| \delta} (C_1 \eps)^{1/2} t^{-3/2} s^{|J_2| \delta}
\\
&\lesssim 
(C_1 \eps)^{3/2} t^{-5/2} s^{|J| \delta}
=
(C_1 \eps)^{3/2} (s/t)^{5/2} s^{-5/2 + |J| \delta},
\endaligned
$$
which finishes the proof of \eqref{eq:comm-001}.

For the proof of \eqref{eq:comm-002}, we recall
$$
[1 - u, \del^I L^J ] v
=
\sum_{\substack{I_1 + I_2 = I \\ |I_1| \geq 1}}
\del^{I_1} u \del^{I_2} v,
$$
and note that there at least one partial derivative hitting on the wave component $u$, which is good.
Next, we proceed in the same way but recall the estimate below from \eqref{eq:D1}
$$
| \del^{I_1} u|
\lesssim
C_1 \eps t^{-1/2} s^{-1},
\quad
1 \leq |I_1| \leq N - 4
$$
which allows us to conclude that
$$
\aligned
|\big( [1 - u, \del^I ] v \big) (t, x)|
&\lesssim
(C_1 \eps)^{3/2} t^{-2} s^{-1}
=
(C_1 \eps)^{3/2} (s/t)^{2} s^{-3}.
\endaligned
$$

\end{proof}

We are in a position to give the proof of the refined sup-norm bounds for the Klein-Gordon component.

\begin{proposition}[Refined pointwise estimates for $v$]
The following estimates are valid
\be 
\big| \del^I L^J v \big| + \big| (t/s) \underdel_\perp \del^I L^J v \big|
\lesssim
C_1 \eps t^{-3/2} s^{|J| \delta},
\quad
|I| + |J| \leq N - 4.
\ee
\end{proposition}
\begin{proof}
We act $\del^I L^J$ on the Klein-Gordon equation in \eqref{eq:model} to get
$$
- \Box \del^I L^J v + (1 - u) \del^I L^J v
=
[1 - u, \del^I L^J] v.
$$

In order to bound the quantity
$$
\big| \del^I L^J v \big| + \big| (t/s) \underdel_\perp \del^I L^J v \big|,
$$
we now apply the pointwise estimates for Klein-Gordon component in Proposition \ref{lem:supKG}.
We need to bound the term $V(t, x)$ in \eqref{eq:supKG1}. According to the definition of $V(t, x)$, we have to bound
$$
F(s)
\leq
\int_{s_0}^s \Big( \big| R[\del^I L^J v](\lambda t/s, \lambda x/s) \big| + \lambda^{3/2} \big| [1 - u, \del^I L^J] v \big|(\lambda t/s, \lambda x/s) \Big) \, d\lambda, 
$$
in which $F(s)$ was defined in \eqref{eq:odesource} in Proposition~\ref{lem:supKG}, and bound the factor
$$
e^{\int_{s_0}^s | {d\over d\lambda} u(\lambda t/s, \lambda x/s)| \, d\lambda}.
$$
Then by recalling the estimate \eqref{eq:est-R} and the commutator estimates \eqref{eq:comm-002} from Lemma \ref{lem:6.4} and Lemma \ref{lem:6.5}, we have
$$
F(s)
\lesssim
C_1 \eps (s/t)^{3/2} s^{|J| \delta},
$$
and on the other hand, Lemma \ref{lem6.3} tells us that
$$
e^{\int_{s_0}^s | {d\over d\lambda} u(\lambda t/s, \lambda x/s)| \, d\lambda}
\lesssim 1.
$$
Again according to Proposition \ref{lem:supKG}, we are led to the desired results
$$
\big| \del^I L^J v \big| + \big| (t/s) \underdel_\perp \del^I L^J v \big|
\lesssim s^{-3/2} V(t, x) 
\lesssim
s^{-3/2} | F | 
\lesssim
C_1 \eps t^{-3/2} s^{|J| \delta}.
$$

The proof is done.
\end{proof}

As a consequence, we have
\bel{eq:delv-high} 
| \del \del^I L^J v |
\lesssim 
C_1 \eps t^{-1/2} s^{-1 + |J| \delta},
\quad
|I| + |J| \leq N - 4,
\ee
which is due to the following two identities (see also \cite{PLF-YM-cmp}):
$$
\del_t
=
{t^2 \over s^2} \big( \underdel_\perp -(x^a / t) \underdel_a \big),
\qquad
\del_a
=
-{t x^a \over s^2} \underdel_\perp + {x^a x^b \over t^2} \underdel_b + \underdel_a.
$$
We note that \eqref{eq:delv-high} is used when we estimate the pointwise decay of the null form $\del_\alpha u \del^\alpha v$ in \eqref{eq:supf} below.


\section{Refined estimates for the wave component}
\label{sec:refinew}

\subsection{Overview of the strategy on treating $u$}

If we deal directly with the nonlinearity $u v$ for the wave equation in \eqref{eq:model}, it is very difficult to get either desired energy estimates or pointwise estimates.
Due to this difficulty, we are motivated to do a transformation and seek for a new unknown which satisfies a wave equation with good nonlinearity, and which meanwhile is close to the original wave component $u$ up a higher order correction term. The idea to treat the Klein-Gordon field is similar as the use of a normal form transformation by Shatah \cite{Shatah} combined with the technique used to deal with $wave$--$wave$ interaction used by Tsutsumi \cite{Tsutsumi}. But before we do the transformation, we find it necessary to first split the wave equation into two, which agrees with the special structure of the equation for $u$.

\begin{proposition}
\label{prop:split}
Let $(u, v)$ be a solution to the model problem \eqref{eq:model}
$$
\aligned
-\Box u 
&= 
u v + u \del_t v,
\\
-\Box v + v 
&= 
u v,
\\
\big( u, \del_t u \big) (2, \cdot)
=
\big( u_0, u_1 \big),
&\quad
\big( v, \del_t v \big) (2, \cdot)
=
\big( v_0, v_1 \big),
\endaligned
$$
then we can split $u$ into the following form
\bel{eq:u-split}
u = U_1 + \del_t U_2,
\ee
in which $U_1$ and $U_2$ are solutions to the two wave equations below:
\bel{eq:U_1} 
\aligned
-\Box U_1
&=
u v - v \del_t u ,
\\
\big(U_1, \del_t U_1 \big) (2, \cdot)
&=
\big( u_0, u_1 + u_0 v_0 \big),
\endaligned
\ee
and 
\bel{eq:U_2} 
\aligned
-\Box U_2
&=
u v,
\\
\big(U_2, \del_t U_2 \big) (2, \cdot)
&=
\big( 0, 0 \big).
\endaligned
\ee
\end{proposition}
We recall that this key observation of splitting as in \eqref{eq:u-split} is due to Katayama \cite{Katayama12a}.

Next, we do a transformation to make the nonlinearities in the $U_1$ and $U_2$ equations easier to deal with.

\begin{proposition}
Consider the wave equations of $U_1$ and $U_2$ in Proposition~\ref{prop:split},
and set 
$$
\tildeU_1 := U_1 + uv,
\qquad
\tildeU_2 := U_2 + uv,
$$  
then the new unknowns $\tildeU_1$ and $\tildeU_2$ satisfy wave equations with new nonlinearities, which are easy to handle, i.e.
\bel{eq:tildeU_1}
-\Box \tildeU_1
=
-\del^\alpha u \del_\alpha v - v \del_t u + u^2 v + u v^2 + u v \del_t v,
\ee
and 
\bel{eq:tildeU_2}
-\Box \tildeU_2
=
-\del^\alpha u \del_\alpha v + u^2 v + u v^2 + u v \del_t v.
\ee
\end{proposition}
\begin{proof}
The proof follows by simple calculations.
We only do it for $\tildeU_2$
$$
-\Box \tildeU_2
=
-\Box (U_2 + u v)
=
-\Box U_2 - \del^\alpha u \del_\alpha v + (-\Box u) v + u (-\Box v + v) - u v,
$$
then by utilising the equations in \eqref{eq:U_2}, we finally arrive at \eqref{eq:tildeU_2}.
\end{proof}

The following consequences follow immediately, which say about that $U$'s are very close to $\tildeU$'s.

\begin{lemma}
Assume $U_1$ and $U_2$ are solutions to \eqref{eq:U_1} and \eqref{eq:U_2} respectively, and let the bootstrap assumptions in \eqref{eq:boots-assum} hold, then the following estimates are verified for all $s \in [2, s_1)$ and $p = 1, 2$.
For $|I| + |J| \leq N$, we have
\bel{eq:U-tildeU-close}
\aligned
{1\over 2} E(s, \del^I L^J U_p)^{1/2} - (C_1 \eps)^{3/2}
\leq
&E(s, \del^I L^J \tildeU_p)^{1/2}
\leq
2 E(s, \del^I L^J U_p)^{1/2} + 2 (C_1 \eps)^{3/2},
\\
{1\over 2} E_{con} (s, \del^I L^J U_p)^{1/2} - (C_1 \eps)^{3/2} s^{1/2}
\leq
&E_{con} (s, \del^I L^J \tildeU_p)^{1/2}
\leq
2 E_{con} (s, \del^I L^J U_p)^{1/2} + 2 (C_1 \eps)^{3/2} s^{1/2}.
\endaligned
\ee
For $|I| + |J| \leq N-4$, we have
\bel{eq:U-tildeU-close02}
 | \del^I L^J (U_p - \tildeU_p) |
\leq
(C_1 \eps)^{3/2} t^{-3/2}.
\ee
\end{lemma}
\begin{proof}
The proof follows by the fact that the difference between $U_p$ and $\tildeU_p$ is a quadratic term $u v$, which has very good decay property.

For easy understanding, we provide the proof for 
$$
E(s, \del^I L^J \tildeU_1)^{1/2}
\leq
2 E(s, \del^I L^J U_1)^{1/2} + 2 (C_1 \eps)^{3/2},
$$
and the proofs for other cases are naturally the same.
Recall that $\tildeU_1 = U_1 + uv$, so it holds
$$
\aligned
E(s, \del^I L^J \tildeU_1)
&= \big\| (s/t) \del_t \del^I L^J \tildeU_1 \big\|_{L^2_f(\Hcal_s)}^2 + \sum_a \big\| (s/t) \underdel_a \del^I L^J \tildeU_1 \big\|_{L^2_f(\Hcal_s)}^2 
\\
&\leq 2 \big\| (s/t) \del_t \del^I L^J U_1 \big\|_{L^2_f(\Hcal_s)}^2 + 2 \sum_a \big\| (s/t) \underdel_a \del^I L^J U_1 \big\|_{L^2_f(\Hcal_s)}^2 
\\
&+
2 \big\| (s/t) \del_t \del^I L^J (u v) \big\|_{L^2_f(\Hcal_s)}^2 + 2 \sum_a \big\| (s/t) \underdel_a \del^I L^J (u v) \big\|_{L^2_f(\Hcal_s)}^2 
\\
&\leq 2 \big\| (s/t) \del_t \del^I L^J U_1 \big\|_{L^2_f(\Hcal_s)}^2 + 2 \sum_a \big\| (s/t) \underdel_a \del^I L^J U_1 \big\|_{L^2_f(\Hcal_s)}^2 
+ 2 (C_1 \eps)^3.
\endaligned
$$
Thus we obtain
$$
E(s, \del^I L^J \tildeU_1)^{1/2}
\leq 2 E(s, \del^I L^J U_1)^{1/2}
+ 2 (C_1 \eps)^{3/2}.
$$

\end{proof}


\subsection{Estimates of the $U_1$ part}

We are now about to derive various estimates for $U_1$, which will be based on the analysis of the new unknown $\tildeU_1$.
We start by a simple lemma, estimating $v \del_t u$.

\begin{lemma}
\label{lem:delu-v}
Let the bootstrap assumptions in \eqref{eq:boots-assum} be true, then it holds
\bel{eq:delu-v}
\big\| \del^I L^J (v \del_t u ) \big\|_{L^2_f(\Hcal_s)}
\lesssim
(C_1 \eps)^{3/2} s^{-3/2 + |J|\delta},
\quad
|I| + |J| \leq N,
\ee
and
\bel{eq:delu-v-sup}
\big| \del^I L^J (v \del_t u ) \big|
\lesssim
(C_1 \eps)^{3/2} t^{-2} s^{-1 + |J| \delta},
\quad
|I| + |J| \leq N - 4.
\ee
\end{lemma}
\begin{proof}
We directly do the estimates
$$
\aligned
\big\| \del^I L^J (v \del_t u ) \big\|_{L^2_f(\Hcal_s)}
&\leq
\sum_{\substack{I_1 + I_2 = I \\ J_1 + J_2 = J}} \big\| \del^{I_1} L^{J_1} \del_t u \del^{I_2} L^{J_2} v \big\|_{L^2_f(\Hcal_s)}
\\
&\leq 
\sum_{\substack{I_1 + I_2 = I, J_1 + J_2 = J \\ |I_1| + |J_1| \leq |I_2| + |J_2|}} \big\| \del^{I_1} L^{J_1} \del_t u \|_{L^\infty(\Hcal_s)} \big\| \del^{I_2} L^{J_2} v \big\|_{L^2_f(\Hcal_s)} 
\\
&+
\sum_{\substack{I_1 + I_2 = I, J_1 + J_2 = J \\ |I_1| + |J_1| \geq |I_2| + |J_2|}} \big\| (s/t) \del^{I_1} L^{J_1} \del_t u \|_{L^2_f(\Hcal_s)} \big\| (t/s) \del^{I_2} L^{J_2} v \big\|_{L^\infty(\Hcal_s)},
\endaligned
$$
and finally the basic estimates in Subsection \ref{subsec:DE} completes the proof of \eqref{eq:delu-v}.

For the sup-norm bound, note that
$$
\big| \del^I L^J (v \del_t u ) \big|
\leq
\sum_{\substack{I_1 + I_2 = I \\ J_1 + J_2 = J}} \big| \del^{I_1} L^{J_1} \del_t u \del^{I_2} L^{J_2} v \big|,
$$
and then it follows from the bootstrap assumptions \eqref{eq:boots-assum} as well as the pointwise estimates \eqref{eq:D1} for $\del^{I_1} L^{J_1} \del_t u$.
\end{proof}

\begin{lemma}
We have
\bel{eq:nonl} 
\aligned
\Big\| \del^I L^J \big( -\del^\alpha u \del_\alpha v - v \del_t u  + u^2 v + u v^2 + u v \del_t v\big) \Big\|_{L^2_f(\Hcal_{s})}
\lesssim
(C_1 \eps)^{3/2} s^{-3/2 + |J|\delta},
\quad
|I| + |J| \leq N,
\endaligned
\ee 
as well as
\bel{eq:supf}
\aligned
\Big| \del^I L^J \big( -\del^\alpha u \del_\alpha v - v\del_t u  + u^2 v + u v^2 + u v \del_t v \big) \Big|
\lesssim
(C_1 \eps)^{3/2} t^{-2} s^{-1 + |J| \delta},
\quad
|I| + |J| \leq N - 4.
\endaligned
\ee
\end{lemma}
\begin{proof}
The terms to be estimated are either null forms, terms of the form $\del^I L^J (\del_t u v)$ or cubic terms. Since $\del^I L^J (\del_t u v)$ is already treated in Lemma~\ref{lem:delu-v} and cubic terms behave very nicely, one refers to Lemma~\ref{lem:null} for more details on treating null forms. 
\end{proof}

\begin{proposition}[Energy estimates for $U_1$]
Consider the wave equation in \eqref{eq:U_1} and assume the bounds in \eqref{eq:boots-assum} hold, then we have the following energy estimates for $U_1$
\be 
E(s, \del^I L^J U_1)^{1/2}
\lesssim  
\eps + (C_1 \eps)^{3/2},
\quad
|I| + |J| \leq N.
\ee
\end{proposition}
\begin{proof}
Firstly, by \eqref{eq:U-tildeU-close}, we know
$$
E(2, \del^I L^J \tildeU_1)^{1/2}
\leq
2 \eps.
$$
Then recall the energy estimates \eqref{eq:w-EE} for wave equations and we easily obtain
$$
\aligned
&\quad 
E(s, \del^I L^J \tildeU_1)^{1/2}
\\
&\leq
E(2, \del^I L^J \tildeU_1)^{1/2}
+
\int_2^s \Big\| \del^I L^J \big( -\del^\alpha u \del_\alpha v - \del_t u v + u^2 v + u v^2 + u v \del_t v \big) \Big\|_{L^2_f(\Hcal_{s'})} \, ds'
\\
&\lesssim
\eps + (C_1 \eps)^{3/2},
\endaligned
$$
in which the last inequality is due the estimate \eqref{eq:nonl}.
By recalling the equivalence relation \eqref{eq:U-tildeU-close} between $U_1$ and $\tildeU_1$ we complete the proof.
\end{proof}

The ideas of the proofs for the two propositions below are very similar to the one above, i.e. we can get good estimates for the auxiliary unknown $\tildeU_1$ easily, and then an application of the equivalence relation \eqref{eq:U-tildeU-close} in turn gives us good estimates of the unknown $U_1$. And we omit the proofs for the following two propositions.

\begin{proposition}[Conformal-type energy estimates for $U_1$]
The conformal-type energy introduced in Subsection \ref{subsc:conf} satisfies
\be 
E_{con} (s, \del^I L^J U_1)^{1/2}
\lesssim  
\eps  + (C_1 \eps)^{3/2} s^{1/2 + |J| \delta},
\quad
|I| + |J| \leq N.
\ee
\end{proposition}

Consequently, we have
\bel{eq:U_1-L2}
\big\| (s/r) \del^I L^J U_1 \big\|_{L^2_f(\Hcal_s)}
\lesssim
\eps  + (C_1 \eps)^{3/2} s^{1/2 + |J| \delta},
\quad
|I| + |J| \leq N,
\ee
which is due to the conformal--type bounds for $U_1$ above and the Hardy--type inequality \eqref{eq:Hardy}.

\begin{proposition}[Pointwise estimates for $U_1$]
We have 
\bel{eq:P-U_1} 
| \del^I L^J U_1 |
\lesssim
\big( \eps + (C_1 \eps)^{3/2} \big) t^{-1} s^{|J| \delta},
\quad
|I| + |J| \leq N - 4.
\ee
\end{proposition}
The proof of this Proposition clearly follows from Lemma \ref{lem:supwave} and the sup-estimate obtained in \eqref{eq:supf}.  


\subsection{Estimates of the $U_2$ part}

We state the following propositions about estimates of $U_2$, but we do not provide proofs as they are either the same as or easier than those of $U_1$.

\begin{proposition}[Energy estimates for $U_2$]
Consider the wave equation in \eqref{eq:U_2} and assume the bounds in \eqref{eq:boots-assum} hold, then we have the following energy estimates for $U_2$
\be 
E(s, \del^I L^J U_2)^{1/2}
\lesssim  
\eps + (C_1 \eps)^{3/2},
\quad
|I| + |J| \leq N.
\ee
\end{proposition}

As a consequence, it gives us
\bel{eq:U_2-L2}
\big\|(s/t) \del_t \del^I L^J U_2 \big\|_{L^2_f(\Hcal_s)} + \big\|(s/t) \del^I L^J \del_t U_2 \big\|_{L^2_f(\Hcal_s)}
\lesssim
\eps + (C_1 \eps)^{3/2},
\quad
|I| + |J| \leq N.
\ee

\begin{proposition}[Pointwise estimates for $U_2$]
We have 
\bel{eq:P-U_2} 
| \del_t \del^I L^J U_2 | + | \del^I L^J \del_t U_2 |
\lesssim
\big( \eps + (C_1 \eps)^{3/2} \big) t^{-1/2} s^{-1},
\quad
|I| + |J| \leq N - 4.
\ee
\end{proposition}
The proof of this Proposition clearly follows from Lemma \ref{lem:supwave} and the Sobolev embedding of Lemma \ref{lem:sobolev}.


\subsection{Refined estimates for $u$}
\label{subsec:Ru}

We are ready to derive the refined estimates for $u$, which will be based on the analysis of the new unknown $U$. To clarify the role played by the $U_p$ and $\tildeU_p$ ($p= 1, 2$) unknowns, we revisit the difficulties in estimating directly the original wave component $u$. Recall that the nonlinearities in the $u$ equation are $F_u = u v + u \del_t v$, and the energy does not decay sufficiently fast, i.e. 
$$
\| F_u \|_{L^2_f(\Hcal_s)}
\lesssim \min \big\{ \|(s/t) u\|_{L^2_f(\Hcal_s)} \| (t/s) v, (t/s) \del_t v \|_{L^\infty(\Hcal_s)}, \|u\|_{L^\infty(\Hcal_s)} \| v, \del_t v \|_{L_f^2(\Hcal_s)} \big\}
\lesssim s^{-1}.
$$
This tells us that closing the bootstrap assumptions on the natural wave energy $E(s, u)^{1/2}$ is critical, and even worse, closing the bootstrap assumptions on the conformal--type wave energy $E_{con}(s, u)^{1/2}$ and $\| (s/t) u \|_{L^2_f(\Hcal_s)}$ is far from possible, see the energy estimate \eqref{eq:con-estimate}. 

So how do the nonlinear transformations help? It is easier to explain if we look the procedure backward. First, we find the wave equations for $\tildeU_p$ have good nonlinearities (\eqref{eq:tildeU_1} and \eqref{eq:tildeU_2}) which are possible to control. Next, the difference between the unknowns $U_p$ and $\tildeU_p$ is a simple quadratic term $u v$, which indicates that all of the estimates valid for $\tildeU_p$ are also true for $U_p$ (more precisely $U_p/2$), and we we can bound $U_p$ using $\tildeU_p$. Finally, we can control the original wave component $u$ by the relation $u = U_1 + \del_t U_2$.

\begin{proposition}[Refined energy estimates for $u$]
Consider the wave equation in \eqref{eq:model} and assume the bounds in \eqref{eq:boots-assum} hold, then we have the following refined ones

\be 
\aligned
E(s, \del^I L^J u)^{1/2}
&\lesssim 
\eps + (C_1 \eps)^{3/2},
&\quad   |I| + |J| \leq N - 1,
\\
E(s, \del^I u)^{1/2}
&\lesssim 
\eps + (C_1 \eps)^{3/2} s^\delta,
&\quad   |I| = N,
\\
E(s, \del^I L^J u)^{1/2}
&\lesssim 
\eps + (C_1 \eps)^{3/2} s^{|J| \delta},
&\quad   |I| + |J| \leq N, |J| \geq 1.
\endaligned
\ee
\end{proposition}
\begin{proof}
For $|I| + |J| \leq N - 1$, we have
$$
E(s, \del^I L^J u)^{1/2}
\lesssim
E(s, \del^I L^J U_1)^{1/2} + E(s, \del^I L^J \del_t U_2)^{1/2},
$$
then the energy estimates of $U_1$ and $U_2$ and the commutators give the desired result.

Next, for the case of $|I| + |J| = N$ with $|J| \geq 1$, we recall the original equation in \eqref{eq:model} and have
\be 
-\Box \del^I L^J u
=
\sum_{\substack{I_1 + I_2 = I \\ J_1 + J_2 = J}} \Big( \del^{I_1} L^{J_1} u \del^{I_2} L^{J_2} v + \del^{I_1} L^{J_1} u \del^{I_2} L^{J_2} \del_t v \Big).
\ee
Then by the energy estimates for wave components \eqref{eq:w-EE}, it is true that
$$
\aligned
&\quad
E(s, \del^I L^J u)^{1/2}
\\
&\leq 
E(2, \del^I L^J u)^{1/2}
+
\int \sum_{\substack{I_1 + I_2 = I \\ J_1 + J_2 = J}} \Big\| \del^{I_1} L^{J_1} u \del^{I_2} L^{J_2} v + \del^{I_1} L^{J_1} u \del^{I_2} L^{J_2} \del_t v \Big\|_{L^2_f(\Hcal_{s'})} \, ds'.
\endaligned
$$
Successively, we arrive at
$$
E(s, \del^I L^J u)^{1/2}
\lesssim 
\eps + (C_1 \eps)^{3/2} s^{|J| \delta},
$$
which is based on the estimates we already have obtained. 
The case of $|I| = N$ can be treated in a similar way, and hence the proof is done.
\end{proof}

\begin{proposition}[Refined $L^2$-type energy estimates for $u$]
It validates that 
\be 
\big\| (s/t) \del^I L^J u \big\|_{L^2_f(\Hcal_s)}
\lesssim  
\eps  + (C_1 \eps)^{3/2} s^{1/2 + |J| \delta},
\quad
|I| + |J| \leq N.
\ee
\end{proposition}
\begin{proof}
We simply have
$$
\big\| (s/t) \del^I L^J u \|_{L^2_f(\Hcal_s)}
\lesssim
\big\| (s/r) \del^I L^J U_1 \big\|_{L^2_f(\Hcal_s)}
+
\big\| (s/t) \del^I L^J \del_t U_2 \big\|_{L^2_f(\Hcal_s)},
$$
and finish the proof by recalling the estimates \eqref{eq:U_1-L2} and \eqref{eq:U_2-L2}.
\end{proof}

\begin{proposition}[Refined pointwise estimates for $u$]
We have 
\be 
| \del^I L^J u |
\lesssim
\big( \eps + (C_1 \eps)^{3/2} \big) t^{-1} s^{|J| \delta},
\quad
|I| + |J| \leq N - 4.
\ee
\end{proposition}
\begin{proof}
It is true that
$$
| \del^I L^J u |
\leq
| \del^I L^J U_1 | + | \del^I L^J \del_t U_2 |,
$$
and the proof is done by the use of \eqref{eq:P-U_1} and \eqref{eq:P-U_2}.
\end{proof}


\section{Proof of the stability result and further remarks}
\label{sec:Close}

\paragraph{Proof of the stability result}

We first close the bootstrap method, which immediately gives the proof of the main theorem.

\begin{proof}[Proof of Theorem \ref{thm:wKG}]
By collecting all of the refined estimates for wave and Klein-Gordon components, which are stated in the propositions in Section~\ref{sec:refineKG} and Subsection~\ref{subsec:Ru}, we choose large $C_1 \gg 1$ and small $\eps \ll1$ such that $C_1 \eps \ll 1$, then we arrive at the desired estimates in \eqref{eq:boots-refine}. Furthermore, as explained at the end of Subsection~\ref{subsect:BA}, we also have provided the proof of Theorem~\ref{thm:wKG}.
\end{proof}

\paragraph{Remarks on the general system of \eqref{eq:general}.}
We claimed in the Introduction that we can deal with more general systems of coupled wave and Klein-Gordon equations  \eqref{eq:general}. Instead of giving the detailed proof, where one is easily lost in long calculations, we provide here an explanation of the key ingredients of the proof. 

On the one hand, recall that the way we treat null nonlinearities in our analysis is through the following (rough) estimate (see \eqref{eq:est-null1})
$$
\big| \del_\alpha u \del^\alpha v \big|
\lesssim
(s/t)^2 |\del u| |\del v|.
$$ 
We note that although $\del u$ does not behave with good decay, the good factor $s/t$ means that the term $(s/t) \del u$ behaves as well as a Klein-Gordon component, in the sense that
$$
\| (s/t) \del u \|_{L^2_f(\Hcal_s)} \lesssim E(s, u)^{1/2},
\qquad
|(s/t) \del u| \lesssim t^{-3/2}.
$$
In the proof of the improved estimates for the wave component $u$, we have encountered (in the $\tildeU_p$ equation) null forms which are standard to handle. Thus our method works for the nonlinearities of the type $Q(v, \del v; v, \del v)$ in the wave equation. Our method and main Theorem have treated nonlinearities of the type $Q(u; v, \del v)$ and thus can deal, more generally, with  nonlinearities of type $Q(u, v, \del v; v, \del v)$ in the wave equation.

 On the other hand, the most difficult step in estimating the Klein-Gordon equation lies in obtaining the sharp pointwise decay for the Klein-Gordon component 
$$
|v| \lesssim t^{-3/2},
$$ 
so that we can initialise the induction on the growth of energies in terms of $|J|$. According to Proposition \ref{lem:supKG}, we can obtain the sharp decay result of $v$ as long as the nonlinearity $F_v$ decays faster that $(s/t)^{-3/2} s^{-5/2 - \delta'}$ with any $\delta' > 0$. Without much work we find 
$$
|Q(\del u, v, \del v; v, \del v)| \lesssim (s/t)^{-3/2} s^{-5/2 - \delta'}.
$$
In our main Theorem we have treated the nonlinearities $uv$ in the Klein-Gordon equation. In addition Tsutsumi has treated $u u$ in the Klein-Gordon equation by a nonlinear transformation (to transform $uu$ into null forms $\del_\alpha u \del^\alpha u$) \cite{Tsutsumi}. Thus putting these analysis together, we find it possible to treat $Q(u, v; u) + Q(\del u, v, \del v; v, \del v)$ in the Klein-Gordon equation.

To conclude, our method applies to the general system \eqref{eq:general}.

\paragraph{Concluding remarks.}
Motivated by the Klein-Gordon-Zakharov system studied in \cite{OTT, Katayama12a, Tsutaya}, other  classes of coupled wave-Klein-Gordon systems from \cite{Katayama12a} and also \cite{PLF-YM-book}, we have studied the system
$$
\aligned
-\Box u 
&= 
u v + u \del_t v,
\\
-\Box v + v 
&= 
u v.
\endaligned
$$
By relying on the strategy introduced in \cite{PLF-YM-cmp}, we obtained global stability results and sharp decay estimates
$$
|u| 
\lesssim t^{-1},
\qquad
|v|
\lesssim t^{-3/2}.
$$
This system is part of a broarder class of systems where one studies nonlinearities with critical exponents and whether this leads to global stability or finite time blow-up. See for example \cite{John2}, \cite{GLS} for the Strauss conjecture of wave equations, \cite{LS, KT} for Klein-Gordon equations, \cite{Tsutsumi} for Dirac-Proca systems. 
We end this article by asking the following questions for possible future work:
\bei
\item
What are the critical cases of nonlinearities for a wave-Klein-Gordon system in general dimensions?
\item
Depending on the critical cases, does the solution to the system exist globally or blow up in finite time? 
\eei


\section{Acknowledgments} The authors are grateful to Philippe G. LeFloch (Sorbonne University) for helpful discussions. Both authors were supported by the Innovative Training Networks (ITN) grant 642768, entitled ModCompShock. This work was completed when ZW visited Sorbonne Universit\'e in 2017 and in 2018. During the final write-up of this work ZW acknowledges support of the Austrian Science Fund (FWF) project P29900-N27 \textit{Geometric Transport
equations and the non-vacuum Einstein-flow}.



\end{document}